\documentclass[11pt]{amsart}
\usepackage{amsmath, color}
\usepackage{amsfonts}
\usepackage{latexsym}
\usepackage{ifpdf}
\usepackage{graphicx,amssymb,lineno}
\usepackage{amssymb}
\usepackage{mathrsfs}
\usepackage{cite}
\usepackage{extarrows}
\usepackage{enumerate}
\usepackage{ulem}
\usepackage{tikz}
\usepackage{float}
\usepackage{scalefnt}
\usetikzlibrary{calc}
\allowdisplaybreaks[4]
\numberwithin{equation}{section}

\newtheorem{proposition}{{\bf Proposition}}[section]

\newtheorem{lemma}{{\bf Lemma}}[section]
\newtheorem{corollary}{{\bf Corollary}}[section]

\date{}
\begin{document}
\title{\Large A note on Cauchy's formula}
\author{Naihuan Jing, Zhijun Li$^{\dagger}$}
\address{Department of Mathematics, North Carolina State University, Raleigh, NC 27695, USA}
\email{jing@ncsu.edu}
\address{School of Mathematics, Huzhou University, Huzhou, Zhejiang 313000, China}
\email{zhijun1010@163.com}
\thanks{{\scriptsize
\hskip -0.6 true cm MSC (2010): Primary: 05E05; Secondary: 17B69, 05E10, 17B65.
\newline $^\dag$ Corresponding author: zhijun1010@163.com
}}
\maketitle
\begin{abstract}
We use the correlation functions of vertex operators to give a proof of Cauchy's formula
\begin{align*}
\prod^K_{i=1}\prod^N_{j=1}(1-x_iy_j)=\sum_{\mu\subseteq [K\times N]}(-1)^{|\mu|}s_{\mu}\{x\}s_{\mu'}\{y\}.
\end{align*}
As an application of the interpretation, we obtain an expansion of $\prod^\infty_{i=1}(1-q^i)^{i-1}$ in terms of half plane partitions.
\end{abstract}

\section{introduction}
The Schur functions form a distinguished orthonormal basis in the ring of symmetric functions 
\cite{Mac1995} with a number of applications, among which the most prominent is perhaps in the representation theory of both the symmetric and general linear groups \cite{W}.
One important identity in the theory is the Cauchy formula
\begin{align}\label{e:Cauchy1}
\prod^{K}_{i=1}\prod^{N}_{j=1}\frac{1}{1-x_iy_j}=\sum_\lambda s_\lambda\{x\}s_\lambda\{y\},
\end{align}
where $s_{\lambda}\{x\}$ is the Schur function in $x_i's$ and the sum is over all partitions $\lambda$ with length $l(\lambda)\leq \min\{K,N\}$.
Foda, Wheeler and Zuparic \cite{FWZ} have used free fermions to study Schur functions and gave a physical interpretation
of the limit of \eqref{e:Cauchy1} using plane partitions, and the underlying algebraic structure is an infinite dimensional Heisenberg algebra
with central charge $1$. This is partly based on the vertex operator approach to symmetric functions \cite{Jing1991, Jing2000}.

In \cite{Kv1987,Wang1997},
charged free bosonic system provides a different Heisenberg algebra with central charge $-1$: $\mathcal{H}=\{h_n\}_{n\in\mathbb{Z}}$ with the commutation relation $[h_m,h_n]=-m\delta_{m,-n}$. On the Fock space $\mathcal V\simeq\mathbb C[h_{-1}, h_{-2}, \ldots]$ (resp. the dual space $\mathcal V^*$)
of the Heisenberg algebra $\mathcal{H}$ generated by the vacuum vector $|0\rangle$ (resp. dual vacuum $\langle 0|$),
we can introduce the fermionic field $\phi(z)$ to obtain a
base $\{|\lambda\rangle\}$ of $\mathcal{V}$ (resp. $\{\langle \lambda|\}$ of $\mathcal V^*$) that satisfies the skew-orthogonality:
\begin{align}
(|\lambda\rangle,|\mu\rangle)=\langle\lambda|\mu\rangle=(-1)^{|\lambda|}\delta_{\lambda,\mu'},
\end{align}
where $\mu'$ is the conjugate of partition $\mu$.

In this paper, we discuss an alternative approach to understand a related Cauchy identity by viewing the
Schur functions as {\it skew-orthogonal basis} of the symmetric functions. With the help of the half-vertex operator $\phi^+(x)=\exp\left(\sum^{\infty}_{n=1}\frac{h_{-n}}{n}x^n\right)$
we revisit/reprove a variant of Cauchy's formula
\begin{align}
\prod^K_{i=1}\prod^N_{j=1}(1-x_iy_j)=\sum_{\mu\subseteq [K\times N]}(-1)^{|\mu|}s_{\mu}\{x\}s_{\mu'}\{y\},
\end{align}
where the sum runs over all partitions $\mu$ with $\mu_1\leq N,~\mu^{\prime}_1\leq K$. The treatment
is completely self-contained and offers new perspective to understand the dynamic procedure of the vertex operator
action (see Prop. \ref{pro1}).

Plane partitions are two-dimensional analogues of ordinary partitions. They naturally appear in many problems of statistical physics and
quantum field theory (see \cite{Bog2005} and references therein). The well-known MacMahon generating function for plane partitions \cite{Mac1995, St1999} is
\begin{align*}
\sum_{\pi}q^{|\pi|}=\prod^{\infty}_{i=1}\frac{1}{(1-q^i)^i},
\end{align*}
where $\pi$ runs over all plane partitions. It has a physical interpretation
via the Schur process and the KP hierarchy \cite{OR2003} as well as the free fermion system \cite{FWZ}. In the same spirit, we introduce half plane partitions, which form a special class of interlacing partition chains (lower triangular part of plane partitions), and use them to give combinatorial interpretations of $\prod^K_{i=1}\prod^N_{j=1}(1-q^{i+j})$ as well as $\prod^{\infty}_{i=1}(1-q^i)^{i-1}$. One of our results is that
\begin{align}
\sum_{\lambda}(-1)^{|\lambda|}\sum_{\{\lambda\rightarrow \pi\}}\sum_{\{\lambda^{\prime}\rightarrow \check{\pi}\}}q^{|\pi|+|\check{\pi}|}=\prod^\infty_{i=1}(1-q^i)^{i-1},
\end{align}
where $\lambda\rightarrow \pi$ (resp. $\lambda^{\prime}\rightarrow \check{\pi}$) runs through all interlacing chains associated with the
half-plane partition $\pi$ (resp. $\check{\pi}$)\footnote{In the remainder, we use $\pi$ to denote a half plane partition.}.

We remark that the vertex operator $\phi(z)$ is in fact a reformulated Bernstein operator (cf. \cite{Zel, Jing1991})
for the Schur functions.
The action of the half-vertex operator $\phi^+(x)$ on Schur functions can be used to derive Macdonald's skew Schur
functions. Bernstein operator can also be formulated in plethystic manner \cite{Garsia, CT, Las, Z, FJK2016}, and another
combinatorial formulation can be found in \cite{HJS, Rosas}.


The paper is organized as follows. In section 2, we consider the charged free bosonic system and study
an infinite-dimensional Heisenberg algebra with negative central charge, which is
different from the traditional treatment (cf. \cite{Jing1991}). We then introduce the field operator $\phi(z)$ to obtain a base of the ring of symmetric functions.
Through the dynamic action of the vertex operator, we show that the Cauchy identity follows naturally. In section 3, we use half plane partitions to express $\prod^\infty_{i=1}(1-q^i)^{i-1}$ by the identities from section 2.

\section{charged free bosons and Cauchy's identities}
Let $\varphi_i, \varphi_i^*$ ($i\in\mathbb Z$) be the charged free bosons satisfying the commutation relations:
\begin{align}\label{B3}
[\varphi_{i},\varphi^{*}_{j}]=\delta_{i,-j},~~~~[\varphi_{i},\varphi_{j}]=[\varphi^{*}_{i},\varphi^{*}_{j}]=0,
\end{align}
where $[A,B]=AB-BA$ is the commutator. Their generating functions are
\begin{align}
\varphi(z)=\sum_{i\in \mathbb{Z}}\varphi_{i}z^{-i-1},~~~~\varphi^{*}(z)=\sum_{i\in \mathbb{Z}}\varphi^{*}_{i}z^{-i}.
\end{align}
\par Let $\mathcal{M}$ (resp. $\mathcal{M}^*$) be the (resp. dual) Fock space generated by the vacuum vector $|0\rangle$ (resp. $\langle0|$) defined by
\begin{align}
\varphi_i|0\rangle=\varphi^{*}_{i+1}|0\rangle=0,~~i\geq 0 ~~(\text{resp.}~\langle0|\varphi_i=\langle0|\varphi^{*}_{i+1}=0,~~i<0).
\end{align}
\par Define the bosonic operators $\displaystyle h_n=\sum^\infty_{i=-\infty}:\varphi_{-i}\varphi^*_{i+n}:$, where the normal ordering $:\ \ :$
moves the factor annihilating $|0\rangle$ to the right. Then
$\{h_n|n\in \mathbb{Z}\}$ generates the Heisenberg
algebra $\mathcal H$ with central charge $-1$ \cite[p7]{Wang1997}:
\begin{align}\label{e:Heis}
[h_m,h_n]=-m\delta_{m,-n}.
\end{align}
For completeness, we verify \eqref{e:Heis} as follows. Note that
$h_n=\sum\limits_{i\geq -n+1}\varphi_{-i}\varphi^*_{i+n}+\sum\limits_{i\leq -n}\varphi^*_{i+n}\varphi_{-i}$.
It follows from $[AB,C]=A[B,C]+[A,C]B$ that
\begin{align*}
[h_i,\varphi_j]=-\varphi_{i+j},~~~~~~~~~[h_i,\varphi^*_j]=\varphi^*_{i+j}.
\end{align*}
Then we have that

\begin{align}
\notag[h_m,h_n]=&\sum_{i\geq -n+1}[h_m,\varphi_{-i}\varphi^*_{i+n}]+\sum_{i\leq -n}[h_m,\varphi^*_{i+n}\varphi_{-i}]\\
\notag=&-\sum_{i\geq -n+1}\varphi_{-i+m}\varphi^*_{i+n}+\sum_{i\geq -n+1}\varphi_{-i}\varphi^*_{i+m+n}\\
\notag&+\sum_{i\leq -n}\varphi^*_{i+m+n}\varphi_{-i}-\sum_{i\leq -n}\varphi^*_{i+n}\varphi_{-i+m}\\
=&-\sum_{i\geq 1}\varphi_{m+n-i}\varphi^*_{i}+\sum_{i\geq m+1}\varphi_{m+n-i}\varphi^*_{i}\\
\notag&+\sum_{i\leq m}\varphi^*_{i}\varphi_{m+n-i}-\sum_{i\leq 0}\varphi^*_i\varphi_{m+n-i}\\
\notag=&\begin{cases}
-\sum\limits^m_{i=1}\varphi_{m+n-i}\varphi^*_{i}+\sum\limits^m_{i=1}\varphi^*_{i}\varphi_{m+n-i}~~~~~~~~~& \text{if}~ m\geq 0,\\
\sum\limits^0_{i=m+1}\varphi_{m+n-i}\varphi^*_{i}-\sum\limits^0_{i=m+1}\varphi^*_{i}\varphi_{m+n-i} &\text{if}~ m< 0,
\end{cases}\\
\notag=&\begin{cases}
-\sum\limits^m_{i=1}[\varphi_{m+n-i},\varphi^*_{i}]~~~~& \text{if}~ m\geq 0,\\
\sum\limits^0_{i=m+1}[\varphi_{m+n-i},\varphi^*_{i}]&\text{if}~ m< 0,
\end{cases}\\
\notag=&-m\delta_{m,-n},
\end{align}
where we have used the convention that $\sum\limits^0_{1}[\varphi_{m+n-i},\varphi^*_{i}]:=0$.

The Fock space $\mathcal{V}$, generated linearly by the left action of $\mathbb{C}[h_{-1},h_{-2},h_{-3},\dots]$ on $|0\rangle$, is a subspace of $\mathcal{M}$. Similarly, $\mathcal{M^*}$ has a subspace $\mathcal{V^*}=\langle0|\mathbb{C}[h_1,h_2,h_3,\dots]$. It is known
that $\mathcal V$ (or $\mathcal V^*$) is the unique left (or right) irreducible representation of the Heisenberg algebra $\mathcal H$.
The following is clear.
\begin{proposition}\label{e:inv}
The charged free bosons carry an 
anti-involution $\omega$ defined by\cite{Kv1987}
\begin{align}
\omega(\varphi_i)=\varphi^*_{-i},~~~~ \omega(\varphi^*_i)=\varphi_{-i}.
\end{align}
Subsequently one has that 
$\omega(h_n)=h_{-n}$.
\end{proposition}

\par A {\it partition} $\lambda=(\lambda_1,\lambda_2,\ldots,\lambda_l)$ of weight $|\lambda|=\sum_i\lambda_i$
is a set of weakly decreasing nonnegative integers. Non-zero $\lambda_i$ are called parts of $\lambda$, and the number of parts is the length of $\lambda$, denoted by $l(\lambda)$. Sometimes we also list the parts in ascending order:
$\lambda=(1^{m_1}2^{m_2}\cdots)$ and define
$z_{\lambda}=\prod_ii^{m_i}m_i!$. The $conjugate$ partition $\lambda^{\prime}$ is defined by
\begin{align}
\lambda^{\prime}_i=\text{Card}\{j:\lambda_j\geq i\}.
\end{align}
In particular, $\lambda^{\prime}_1=l(\lambda)$ and $|\lambda^{\prime}|=|\lambda|$. Pick the rectangle $[N\times M]$ containing the Young diagram of $\lambda$, i.e., $\lambda_1\leq M,~\lambda^{\prime}_1\leq N$, for which we often write $\lambda\subseteq [N\times M]$. In particular, $\lambda\subset[N\times\infty]$ means the set of partitions $\lambda$ with $l(\lambda)\leq N$. Let $\mathcal P$ be the
set of partitions. A partition $\mu=(\mu_{1},\dots,\mu_{l+1})$ is said to {\it interlace} the partition $\lambda=(\lambda_{1},\dots,\lambda_{l})$, written as
$\mu\succ \lambda$, if
\begin{align}
\mu_{i}\geq\lambda_{i}\geq\mu_{i+1}
\end{align}
for all $1\leq i\leq l$. As a result $\mu\geq\lambda$ in the dominance order.

Choose the normalization $\langle 0|1|0 \rangle=1$, and define the inner product of $x|0 \rangle,y|0 \rangle\in \mathcal{V}$ via
\begin{align}
(x|0 \rangle,y|0 \rangle)=\langle 0|\omega(x)y|0 \rangle,
\end{align}
and extend bilinearly to the whole space. Thus $(h_{-\lambda}|0 \rangle, h_{-\mu}|0 \rangle)=(-1)^{l(\lambda)}\delta_{\lambda,\mu}z_{\lambda}$,
where $h_{-\lambda}=h_{-\lambda_1}h_{-\lambda_2}\cdots$.

Let $\Lambda=\mathbb Q[x_1, x_2, \ldots ]^{S_{\infty}}$ be the ring of symmetric functions in the $x_n$.
For each integer $k\geq 0$, we define the $complete~symmetric~function$ $s_{k}(x)$ \cite{Mac1995} in infinitely many variables $x_1,x_2,\dots$ by the generating function
\begin{align*}
\sum_{k=0}^{\infty}s_{k}(x)z^{k}=\prod^{\infty}_{i=1}\frac{1}{1-x_iz}.
\end{align*}
\par For convenience, set $s_{-k}(x)=0$ for $k>0$. To each partition $\lambda$ we define the $Schur~function$ $s_\lambda(x)$ by the Jacobi-Trudi formula \cite{Mac1995}
\begin{align*}
s_{\lambda}(x)=\det\big{(}s_{\lambda_{i}-i+j}(x)\big{)}_{1\leq i,j\leq l(\lambda)}.
\end{align*}
It is well-known that 
$$\Lambda_{\mathbb Z}=\mathbb Z[s_1(x), s_2(x), \ldots]=\sum\limits_{\lambda\in\mathcal P}\mathbb Zs_{\lambda}(x).$$

\par For the rest of the paper, we denote by $s_\lambda\{x\}$ the Schur function
in finitely many variables $\{x\}=\{x_1,x_2,\cdots ,x_K\}$.
It is known that \cite[(3.1),(5.9),(5.11)]{Mac1995}
\begin{align}\label{bc3}
&s_\mu\{x\}=\sum_{\nu\prec \mu}s_\nu\{\bar{x}\}x^{|\mu|-|\nu|}_K, \\
\label{bc11}&s_\mu\{x\}=0, \qquad l(\mu)>K,
\end{align}
where  $\{x\}=\{x_1,\cdots, x_K\}, \{\bar{x}\}=\{x\}\backslash \{x_K\} $.

Introduce the vertex operator (cf. \cite{FLM} for general information)
\begin{align*}
\phi(z)=\sum_{i\in\mathbb{Z}}\phi_iz^{-i}=\phi^+(z)\phi^-(z^{-1})=
\exp\left(\sum^\infty_{n=1}\frac{h_{-n}}{n}z^n\right)\exp\left(\sum^\infty_{n=1}\frac{h_n}{n}z^{-n}\right),
\end{align*}
where $\phi^{\pm}(z)=\exp\left(\sum^\infty_{n=1}\frac{h_{\mp n}}{n}z^n\right)$. 
Then by Prop. \ref{e:inv}, 
\begin{align*}
\omega(\phi^{\pm}(z))=\phi^{\mp}(z), \qquad \omega(\phi(z))=\phi(z^{-1}),
\end{align*}
i.e., $\omega(\phi_i)=\phi_{-i}$. Clearly $[\phi^{\pm}(z), \phi^{\pm}(w)]=0$. It follows from direct vertex operator calculation that for $|zw|<1$
\begin{align}\label{e:3}
\phi^-(z)\phi^+(w)=(1-zw)\phi^+(w)\phi^-(z).
\end{align}

We also have
\begin{align}\label{e:ac1}
\langle 0|\phi_{-n}=\phi_n|0\rangle=0,~~n> 0.
\end{align}
For partition $\lambda=(\lambda_1,\lambda_2,\dots,\lambda_l)$, we denote $|\lambda\rangle=\phi_{-\lambda_1}\cdots\phi_{-\lambda_l}|0\rangle$ and $\langle\lambda|=\langle 0|\phi_{\lambda_l}\cdots\phi_{\lambda_1}$. We also define the element $\chi_m$ by the generating function
\begin{align}
\phi^+(z)=\exp\left(\sum^\infty_{n=1}\frac{h_{-n}}{n}z^n\right)=\sum_{m=0}^{\infty}\chi_mz^m.
\end{align}
And for partition $\lambda=(\lambda_1,\lambda_2,\cdots,\lambda_l)$, we define the Schur element $\chi_{\lambda}|0\rangle\in \mathcal V$:
\begin{align}\label{e:JTF2}
\chi_{\lambda}=\det(\chi_{\lambda_i-i+j})_{1\leq i,j\leq l}.
\end{align}
 Observe that $\chi_{\lambda}$ makes sense even if $\lambda$ is a composition. However $\chi_{\lambda}=0$ if $\lambda+\delta=(\lambda_1+l-1, \lambda_2+l-2, \dots, \lambda_l)$ has equal parts
by the determinant property \cite{JR2016}. If $\lambda=\sigma(\mu+\delta)-\delta$ for a partition $\mu$, then $\chi_{\lambda}=\varepsilon(\sigma)\chi_{\mu}$.

We remark that $|\lambda\rangle$ (or $\langle\lambda|$) are Schur basis elements in $\mathcal V$ (or $\mathcal V^*$). In fact,
$\Lambda\simeq \mathbb Q[h_{-1}, h_{-2}, \cdots]$ under
the map $s_n\mapsto \chi_{n}$ \cite{Jing1991, Jing2000}. Therefore $\Lambda_{\mathbb C}\simeq \mathcal V$ (or $\mathcal V^*$) under the identification, and $s_{\lambda} \simeq|\lambda\rangle$ (or $ \langle\lambda|$). For more details on the vertex operator approach to symmetric functions, see \cite{Jing1991}.
Nevertheless, the following discussion is independent from this identification or motivation.
\begin{proposition} One has that for $i, j\in\mathbb Z$
\begin{align}\label{eq14}
\phi_i\phi_j+\phi_{j+1}\phi_{i-1}=0.
\end{align}
\end{proposition}
\begin{proof} By \eqref{e:3} it follows that for $|z|>|w|$
\begin{align}\label{e:contr}
\phi(z)\phi(w)&=(1-\frac{w}{z})\phi^+(z)\phi^+(w)\phi^-(z^{-1})\phi^-(w^{-1})
\end{align}
thus
\begin{align}
z\phi(z)\phi(w)+w\phi(w)\phi(z)=0.
\end{align}
The proposition follows by taking the coefficients.
\end{proof}
\begin{proposition}\label{pro2} For each partition $\lambda$, one has that
\begin{equation}
|\lambda\rangle=\chi_{\lambda}|0\rangle, \qquad \langle\lambda|=\langle 0|\omega(\chi_{\lambda}).
\end{equation}
Moreover, 
$\{|\lambda\rangle\}_{\lambda\in\mathcal P}$ and $\{(-1)^{|\lambda|}|\lambda'\rangle\}_{\lambda\in\mathcal P}$ are (dual) bases of
$\mathcal V$ and $\mathcal V^*$ respectively, i.e.,
\begin{align}\label{e:ortho1}
\left(|\lambda\rangle,|\mu\rangle\right)=\langle \lambda|\mu\rangle=(-1)^{|\lambda|}\delta_{\lambda^{\prime},\mu}.
\end{align}
\end{proposition}
\begin{proof}
Using the method in \cite{Jing1991}, for any composition $\lambda=(\lambda_1,\dots,\lambda_l)\in \mathbb Z_+^l$, it follows from
\eqref{e:contr} and the Vandermonde determinant that
\begin{align}\notag
&\phi_{-\lambda_1}\cdots\phi_{-\lambda_l}|0\rangle=\text{Res}_{z}z^{-\lambda_1-1}_1\cdots z^{-\lambda_l-1}_l\phi(z_1)\cdots \phi(z_l)|0\rangle\\ \notag
=&\text{Res}_{z}z^{-\lambda_1-l}_1z^{-\lambda_2-l+1}_2\cdots z^{-\lambda_l-1}_l\prod_{1\leq i< j\leq l}(z_i-z_j)\exp\left(\sum^\infty_{n=1}\frac{z_1^n+\cdots+z_l^n}{n}h_{-n}\right)|0\rangle\\ \notag
=&\text{Res}_{z}z^{-\lambda_1-l}_1z^{-\lambda_2-l+1}_2\cdots z^{-\lambda_l-1}_l\sum_{\sigma\in S_l}\varepsilon(\sigma)z^{\sigma(l)-1}_1\cdots z^{\sigma(1)-1}_l\exp\left(\sum^\infty_{n=1}\frac{z_1^n+\cdots+z_l^n}{n}h_{-n}\right)|0\rangle\\ \notag
=&\text{Res}_{z}\sum_{\sigma\in S_l}\varepsilon(\sigma)z^{-\lambda_1+\sigma(l)-l-1}_1z^{-\lambda_2+\sigma(l-1)-l}_2\cdots z^{-\lambda_l+\sigma(1)-2}_l\exp\left(\sum^\infty_{n=1}\frac{z_1^n+\cdots+z_l^n}{n}h_{-n}\right)|0\rangle\\ \label{e:base1}
=&\sum_{\sigma\in S_l}\varepsilon(\sigma)\chi_{\lambda_1-\sigma(l)+l}\chi_{\lambda_2-\sigma(l-1)+l-1}\cdots \chi_{\lambda_l-\sigma(1)+1}|0\rangle
=\chi_\lambda|0\rangle,
\end{align}
where $\text{Res}_{z}f(z_1,\dots,z_l)$ denotes the coefficient of $z^{-1}_1\cdots z^{-1}_l$.

For two partitions $\lambda=(\lambda_1, \dots,\lambda_l)$ and $\mu=(\mu_1, \dots,\mu_k)$, we compute by using the Vandermonde
determinant in variables $z_1, \dots, z_{k+l}$:
\begin{align*}
&\left(|\lambda\rangle,|\mu\rangle\right)
=\langle 0|\phi_{\lambda_l}\phi_{\lambda_{l-1}}\cdots \phi_{\lambda_1}\phi_{-\mu_1}\phi_{-\mu_2}\cdots \phi_{-\mu_k}|0\rangle\\
=&\text{Res}_{z}z^{\lambda_l-1}_1z^{\lambda_{l-1}-1}_2\cdots z^{\lambda_1-1}_lz^{-\mu_1-1}_{l+1}z^{-\mu_2-1}_{l+2}\cdots z^{-\mu_k-1}_{l+k}\langle 0|\phi(z_1)\cdots \phi(z_{l+k})|0\rangle\\
=&\text{Res}_{z}z^{\lambda_l-1}_1z^{\lambda_{l-1}-1}_2\cdots z^{\lambda_1-1}_lz^{-\mu_1-1}_{l+1}z^{-\mu_2-1}_{l+2}\cdots z^{-\mu_k-1}_{l+k}\prod_{1\leq i< j\leq l+k}(1-\frac{z_j}{z_i})\\
=&\text{Res}_{z}\sum_{\sigma\in S_{l+k}}\varepsilon(\sigma)z^{\lambda_l+\sigma(l+k)-l-k-1}_1\cdots z^{\lambda_1+\sigma(k+1)-k-2}_l
z^{-\mu_1+\sigma(k)-k-1}_{l+1}\cdots z^{-\mu_k+\sigma(1)-2}_{l+k}\\
=&\varepsilon(\sigma)\delta_{\lambda_l,l+k-\sigma(l+k)}\cdots \delta_{\lambda_1,k+1-\sigma(k+1)}\delta_{\mu_1,\sigma(k)-k}\cdots \delta_{\mu_k,\sigma(1)-1}\\
=&\varepsilon(\sigma)\delta_{\sigma(l+k),k+l-\lambda_l}\cdots \delta_{\sigma(k+1),k+1-\lambda_1}\delta_{\sigma(k),\mu_1+k}\cdots \delta_{\sigma(1),\mu_k+1},
\end{align*}
for some $\sigma\in S_{k+l}$.
So the inner product is nonzero if and only if 
\begin{align}\label{e:2part}
\{k+l-\lambda_l,k+l-1-\lambda_{l-1},\dots,k+1-\lambda_1,\mu_1+k,\dots,\mu_k+1\}\overset{\sigma}\longleftrightarrow\{l+k,l+k-1,\dots,1\}.
\end{align}
\par We claim that \eqref{e:2part} implies that
$\lambda=\mu^{\prime}$ and $\varepsilon(\sigma)=(-1)^{|\lambda|}$. Assume \eqref{e:2part} holds.  By summing the elements in both sets we have that $|\lambda|=|\mu|$.
Also it is easily seen that $l\geq \mu_1$ and $k\geq \lambda_1$. Clearly $k+i-\lambda_i\neq \mu_j+k-j+1$ or
\begin{equation}\label{eq12}
\lambda_i+\mu_j\neq i+j-1
\end{equation}
for all $1\leq i\leq l, 1\leq j\leq k$.
 Suppose $\lambda\neq\mu^{\prime}$, then there exists $1\leq i\leq l$ such that
$\lambda_j=\mu^{\prime}_{j},~1\leq j\leq i-1$ and $\lambda_i\neq\mu^{\prime}_{i}$.
 If $\lambda_i>\mu^{\prime}_{i}$, then $\mu_{\lambda_i}=i-1$ by looking at the diagrams of $\lambda$ and $\mu'$, thus
$\lambda_i+\mu_{\lambda_i}=i+\lambda_i-1$,
which contradicts \eqref{eq12}.
If $\lambda_i<\mu^{\prime}_{i}$, then there exist $i<j$ such that $
\lambda_s\leq\mu^{\prime}_{s}$ for $i\leq s<j$ and $\lambda_j>\mu^{\prime}_{j}$ due to $|\lambda|=|\mu'|$.
Then $\mu_{\lambda_j}=j-1$, and $\lambda_j+\mu_{\lambda_j}=j+\lambda_j-1$,
which violates \eqref{eq12} again. Therefore $\lambda=\mu^{\prime}$.

Now we show that the coefficient of $z^{-\lambda_l}_1\cdots z^{-\lambda_1}_lz^{\mu_1}_{l+1}\cdots z^{\mu_k}_{l+k}$ in $\prod_{1\leq i< j\leq l+k}(1-\frac{z_j}{z_i})$ is $(-1)^{|\lambda|}$. First, since $\lambda_1=\mu^{\prime}_1=k$, for all $l+1\leq j\leq l+k$, $-\frac{z_{j}}{z_l}$
appear in the expansion, and for all $1\leq i\leq l-1$, $-\frac{z_l}{z_i}$ doesn't exist in the product. Note that
\begin{align*}
z^{-\lambda_l}_1\cdots z^{-\lambda_1}_lz^{\mu_1}_{l+1}\cdots z^{\mu_k}_{l+k}=z^{-\lambda_1}_lz_{l+1}z_{l+2}\cdots z_{l+k}z^{-\lambda_l}_1\cdots z^{-\lambda_2}_{l-1}z^{\mu_1-1}_{l+1}\cdots z^{\mu_{\mu^{\prime}_2}-1}_{l+\mu^{\prime}_2}.
\end{align*}
Similarly, $\lambda_2=\mu^{\prime}_2$ implies that for all $l+1\leq j\leq l+\mu^{\prime}_2$, $-\frac{z_j}{z_{l-1}}$ exist
in the product, and for all $1\leq i\leq l-2$, $-\frac{z_{l-1}}{z_i}$ don't appear in the expansion. Continuing the process, we see that the coefficient of $z^{-\lambda_l}_1\cdots z^{-\lambda_1}_lz^{\mu_1}_{l+1}\cdots z^{\mu_k}_{l+k}$ in $\prod_{1\leq i< j\leq l+k}(1-\frac{z_j}{z_i})$ is $(-1)^{|\lambda|}$.

Using \eqref{e:base1} and the remark after \eqref{e:JTF2}, $|\lambda\rangle$'s ($\lambda\in\mathcal P$) span the space $\mathcal V$ and the inner product result shows
that $\{|\lambda\rangle\}$ is a base of $\mathcal{V}$.
\end{proof}

We remark that the coefficient of $z_1^{\lambda_1}\cdots z_l^{\lambda_l}z_{l+1}^{-\lambda'_1}\cdots z_{l+k}^{-\lambda'_k}$ in the
product $\prod_{1\leq i< j\leq l+k}(1-\frac{z_j}{z_i})$ equals to the sign of the shuffle \eqref{e:2part}, which is
$(-1)^{|\lambda|}$. Also \eqref{e:2part} generalizes the well-known combinatorial fact in \cite[(1.7)]{Mac1995} and our proof
thus offers another algebraic one.


\begin{proposition}\label{pro1} For any partition $\lambda$,
one has that
\begin{align}
\label{e:1}&\phi^+(x)|\lambda\rangle=\sum_{\lambda\prec\mu\subseteq [(l(\lambda)+1)\times \infty]}x^{|\mu|-|\lambda|}|\mu\rangle,\\
&\langle\lambda| \phi^-(x)=\sum_{\lambda\prec\mu\subseteq [(l(\lambda)+1)\times\infty]}x^{|\mu|-|\lambda|}\langle\mu|.
\end{align}
\end{proposition}
\begin{proof} It follows from \eqref{e:3} that
\begin{align}
\phi^+(x)\phi(z)=(1-\frac{x}{z})^{-1}\phi(z)\phi^+(x).
\end{align}
\par Note that $\phi^+(x)|0\rangle=\phi(x)|0\rangle$, then 
\begin{align}\label{eq13}
\phi^+(x)\phi(z_1)\phi(z_2)\dots \phi(z_l)|0\rangle=\prod^l_{i=1}(1-\frac{x}{z_i})^{-1}\phi(z_1)\phi(z_2)\dots \phi(z_l)\phi(x)|0\rangle.
\end{align}

Now $\phi^+(x)|\lambda\rangle$ is the coefficient $C$ of $z^{\lambda_1}_1\cdots z^{\lambda_l}_l$ in \eqref{eq13}, and we claim that
\begin{align}\label{e:coeff1} 
C=
\sum_{n_1\geq 0,0\leq n_i\leq \lambda_{i-1}-\lambda_i,2\leq i\leq l+1}x^{n_1+\dots+n_l+n_{l+1}}\phi_{-\lambda_1-n_1}\phi_{-\lambda_2-n_2}\cdots\phi_{-\lambda_l-n_l}\phi_{-n_{l+1}}|0\rangle,
\end{align}
where $\lambda_{l+1}=0$.
By \eqref{eq14}, for any fixed $m,n\in\mathbb Z$ we have that\footnote{Due to the fact \eqref{e:ac1}, the action of $\sum_{i\geq m, j\geq m+1}\phi_{-i}\phi_{-j}x^{i+j+n}$ on $|\lambda\rangle$ is a finite sum.}
\begin{equation}\label{e:quad}
\sum_{i\geq m, j\geq m+1}\phi_{-i}\phi_{-j}x^{i+j+n}=0, \qquad \sum_{i\geq m, j\geq m-1}\phi_{i}\phi_{j}x^{-i-j+n}=0.
\end{equation}
For $\lambda_i\geq \lambda_j$, we have that
\begin{align}\notag
\sum_{n_i\geq 0, n_j\geq 0}\phi_{-\lambda_i-n_i}\phi_{-\lambda_j-n_j}x^{n_i+n_j}&=
\left(\sum_{n_i\geq 0, n_j\geq \lambda_i-\lambda_j+1}+\sum_{n_i\geq 0, 0\leq n_j\leq \lambda_i-\lambda_j}\right)\phi_{-\lambda_i-n_i}\phi_{-\lambda_j-n_j}x^{n_i+n_j}\\ \label{e:quad2}
&=\sum_{n_i\geq 0, 0\leq n_j\leq \lambda_i-\lambda_j}\phi_{-\lambda_i-n_i}\phi_{-\lambda_j-n_j}x^{n_i+n_j}.
\end{align}
In other words,
the first identity of \eqref{e:quad} can be used to trim the summation in \eqref{e:coeff1}.
By definition $C=\sum_{n_i\geq 0}x^{n_1+\cdots+n_{l+1}}\phi_{-\lambda_1-n_1}\cdots \phi_{-\lambda_{l+1}-n_{l+1}}|0\rangle$.
Successive application of \eqref{e:quad2} to the factors from right to left implies \eqref{e:coeff1}. The summation indices of
\eqref{e:coeff1} satisfy that
\begin{align}
\lambda_{i}\geq \lambda_{i+1}+n_{i+1},~~1\leq i\leq l\quad (\lambda_{l+1}=0).
\end{align}
Then $(\lambda_i+n_i)-(\lambda_{i+1}+n_{i+1})\geq n_i\geq 0$, and
$\mu=(\lambda_1+n_1,\dots,\lambda_l+n_l,n_{l+1})$ is a partition that interlaces $\lambda$: $\lambda\prec\mu$. On the other hand, given
$\lambda\prec\mu$, then $n_i=\mu_i-\lambda_i\geq 0$ corresponds to a term in \eqref{e:coeff1}. In summary we have shown that
\begin{align}
C=\sum_{\lambda\prec\mu}x^{|\mu|-|\lambda|}|\mu\rangle.
\end{align}
\end{proof}
\begin{proposition}\label{pro3}
One has the following equations:
\begin{align}
\label{eq5}&\phi^+(x_1)\phi^+(x_2)\dots \phi^+(x_K)|0\rangle=\sum_{\mu\subseteq [K\times\infty]}s_{\mu}\{x\}|\mu\rangle,\\
\label{eq6}&\langle 0|\phi^-(y_1)\phi^-(y_2)\dots \phi^-(y_N)=\sum_{\rho\subseteq [N\times\infty]}s_{\rho}\{y\}\langle \rho|,
\end{align}
where $\{x\}=\{x_1,\cdots, x_K\}$ and $\{y\}=\{y_1,\cdots, y_N\}$.
\end{proposition}
\begin{proof} We argue by induction on $K$. First \eqref{eq5} holds for $K=1$ by
\eqref{e:1} with $|\lambda\rangle=|0\rangle$.
Assume \eqref{eq5} holds for $K-1$, then for $\{x\}=\{x_1, \dots, x_K\}$ and $\{\bar{x}\}=\{x\}\backslash\{x_K\}$
\begin{align*}
\phi^+(x_1)\phi^+(x_2)\dots \phi^+(x_K)|0\rangle=&\phi^+(x_K)\phi^+(x_1)\dots \phi^+(x_{K-1})|0\rangle\\
=&\sum_{\nu\subseteq [(K-1)\times\infty]}s_{\nu}\{\bar{x}\}\phi^+(x_K)|\nu\rangle\\
=&\sum_{\nu\subseteq [(K-1)\times\infty]}s_{\nu}\{\bar{x}\}\sum_{\nu\prec\mu\subseteq [K,\infty)}x^{|\mu|-|\nu|}_K|\mu\rangle\\
=&\sum_{\mu\subseteq [K\times\infty]}s_{\mu}\{x\}|\mu\rangle,
\end{align*}
where we have used \eqref{bc3} in the last equation. 
\end{proof}
Combining Proposition \ref{pro2} with Proposition \ref{pro3}, we obtain the following:
\begin{corollary}\label{co2}
The correlation function $\langle 0|\phi^-(x_1)\phi^-(x_2)\dots \phi^-(x_K)\phi^+(y_1)\phi^+(y_2)\dots \phi^+(y_N)|0\rangle$ has the following two expressions
\begin{align}
&\langle 0|\phi^-(x_1)\phi^-(x_2)\dots \phi^-(x_K)\phi^+(y_1)\phi^+(y_2)\dots \phi^+(y_N)|0\rangle=\prod^K_{i=1}\prod^N_{j=1}(1-x_iy_j),\\
&\langle 0|\phi^-(x_1)\phi^-(x_2)\dots \phi^-(x_K)\phi^+(y_1)\phi^+(y_2)\dots \phi^+(y_N)|0\rangle=\sum_{\mu\subseteq [K\times N]}(-1)^{|\mu|}s_{\mu}\{x\}s_{\mu'}\{y\},
\end{align}
which immediately implies Cauchy's formula
\begin{align}
\label{eq9}\prod^K_{i=1}\prod^N_{j=1}(1-x_iy_j)=\sum_{\mu\subseteq [K\times N]}(-1)^{|\mu|}s_{\mu}\{x\}s_{\mu'}\{y\}.
\end{align}
\end{corollary}
Taking the limits $N\rightarrow \infty,~K\rightarrow \infty $, we obtain the Cauchy identity \cite{Mac1995}:
\begin{align}
\label{eq10}\prod^\infty_{i,j=1}(1-x_iy_j)=\sum_{\mu\in\mathcal P}(-1)^{|\mu|}s_{\mu}(x)s_{\mu'}(y).
\end{align}

\section{half plane partitions and Cauchy's identities}

A {\it half plane partition} $\pi$ is a set of finitely many nonzero integers $\pi(i,j)$ that are weakly bi-decreasing:
$\pi(i,j)\geq\pi(i+1,j), \pi(i,j)\geq\pi(i,j+1)$
for all $i\geq j\geq 1$ with the additional condition
\begin{align}
\pi(i,j)=0,~~i<j.
\end{align}
For convenience, one may add strings of zeros to $\pi(i, j)$ for $i\geq j\gg 0$.
The weight of $\pi$ is $|\pi|=\sum_{i,j\geq 1}\pi(i,j)$. The height $h(\pi)$ of $\pi$ is the maximal $i$ such that $\pi(i,1)>0$.
One also uses the notion of $tableau$, whereby the non-negative integer $\pi(i,j)$ is placed in
row $i$ and column $j$ for any $i\geq j\geq 1$.
\begin{figure}[H]
  \centering
\scalefont{0.8}
  \begin{tikzpicture}[scale=0.6]
   \coordinate (Origin)   at (0,0);
    \coordinate (XAxisMin) at (0,0);
    \coordinate (XAxisMax) at (3,0);
    \coordinate (YAxisMin) at (0,0);
    \coordinate (YAxisMax) at (0,4);
    \draw [thin, black] (0,0) -- (1,0);
    \draw [thin, black] (0,1) -- (3,1);
    \draw [thin, black] (0,2) -- (3,2);
    \draw [thin, black] (0,3) -- (2,3);
    \draw [thin, black] (0,4) -- (1,4);
    \draw [thin, black] (0,0) -- (0,4);
    \draw [thin, black] (1,0) -- (1,4);
    \draw [thin, black] (2,1) -- (2,3);
    \draw [thin, black] (3,1) -- (3,2);
    \node[inner sep=2pt] at (0.5,0.5) {$2$};\node[inner sep=2pt] at (0.5,1.5) {$3$}; \node[inner sep=2pt] at (1.5,1.5) {$2$}; \node[inner sep=2pt] at (2.5,1.5) {$1$};\node[inner sep=2pt] at (0.5,2.5) {$4$};\node[inner sep=2pt] at (1.5,2.5) {$3$};\node[inner sep=2pt] at (0.5,3.5) {$5$};
    \end{tikzpicture}
    \caption{Tableau representation of a half plane partition. e.g. the top and third rows correspond to the boxes $(1,1)$
     and $(3,1),~(3,2),~(3,3)$ respectively.}
\end{figure}
Let $\pi$ be a half plane partition. For $i\geq 0$, define the partition $\pi_i$, called  a {\it diagonal~slice} of $\pi$, with the parts given by
\begin{align}
(\pi_i)_j=\pi(j+i,j), \qquad j\geq 1.
\end{align}
For the half plane partition in Figure 1, the diagonal slices are given as follows:
\begin{align}\label{e:slice}
\pi_0=(5,3,1),~~\pi_1=(4,2),~~\pi_2=(3),~~\pi_3=(2),
\end{align}
where $|\pi|=5+3+1+4+2+3+2=20,~h(\pi)=4$.

A half plane partition can be considered as a lower triangular part of a plane partition\cite{Mac1995}. It is known that plane partitions and interlacing partitions are closely related. The following fact
was due to Okounkov and Reshetikhin \cite{OR2003} for general plane partitions.
\begin{lemma}\label{le1}
Let $\pi_i$ be the $diagonal~slices$ of the half plane partition $\pi$. Then one has
\begin{align}
\pi_i\succ \pi_{i+1},~~i\geq 0.
\end{align}
\end{lemma}

If $\lambda$ is a partition, an {\it interlacing partition chain} of $\lambda$ is a series of partitions starting from
 $\lambda$ and ending at $\emptyset$:
 \begin{align*}
 \emptyset=\lambda^{(n)}\prec\dots \prec \lambda^{(1)}\prec \lambda^{(0)}=\lambda.
 \end{align*}
 For a partition $\lambda$, let $\{\lambda\rightarrow T\}$ be the set of all interlacing partition chains $T$ of $\lambda$.
By Lemma \ref{le1} each half plane partition $\pi$ canonically gives rise to
an interlacing partition chain of $\pi_0$.
For example, the interlacing partition chain of $\pi_0=(5,3,1)$ in Figure 1 is
\begin{align}
\emptyset=\pi_4\prec\pi_3=(2)\prec\pi_2=(3)\prec\pi_1=(4,2)\prec \pi_0=(5,3,1).
\end{align}
Denote by $\{\lambda\rightarrow \pi\}$ the set of half plane partitions initiating at $\pi_0=\lambda$, and
$\{\lambda\rightarrow \pi\}_n$ the set of half plane partitions with $\pi_0=\lambda,~\pi_n=\emptyset$. Thus we have
\begin{align}
\{\lambda\rightarrow \pi\}_1\subseteq \{\lambda\rightarrow \pi\}_2\subseteq\{\lambda\rightarrow \pi\}_3\subseteq\dots
\end{align}
and $\lim_{n\rightarrow\infty}\{\lambda\rightarrow \pi\}_n=\{\lambda\rightarrow \pi\}$.
\par Proposition \ref{pro1} gives that
\begin{align}
\label{eq7}&\phi^+(y_1)\dots \phi^+(y_N)|0\rangle=\sum_{\emptyset=\check{\pi}_N\prec\dots \prec\check{\pi}_0=\mu\subseteq [N\times\infty]}\prod^N_{i=1}y_i^{|\check{\pi}_{N-i}|-|\check{\pi}_{N-i+1}|}|\mu\rangle,\\
\label{eq8}&\langle 0|\phi^-(x_1)\dots \phi^-(x_K)=\sum_{\emptyset=\pi_K\prec\dots \prec\pi_0=\lambda\subseteq [K\times\infty]}\prod^K_{i=1}x_i^{|\pi_{K-i}|-|\pi_{K-i+1}|}\langle\lambda|.
\end{align}
Taking the $q$-specialization $x_i=q^{K-i+1}$ and $y_j=q^{N-j+1}$, one has
\begin{align}
\notag&\prod^K_{i=1}x_i^{|\pi_{K-i}|-|\pi_{K-i+1}|}=q^{|\pi_0|-|\pi_1|}q^{2(|\pi_1|-|\pi_2|)}\dots q^{K(|\pi_{K-1}|-|\pi_K|)}=q^{|\pi_0|+|\pi_1|+\dots+|\pi_K|}=q^{|\pi|},\\
\notag&\prod^N_{i=1}y_i^{|\check{\pi}_{N-i}|-|\check{\pi}_{N-i+1}|}=q^{|\check{\pi}_0|-|\check{\pi}_1|}q^{2(|\check{\pi}_1|-|\check{\pi}_2|)}\dots q^{N(|\check{\pi}_{N-1}|-|\check{\pi}_N|)}=q^{|\check{\pi}_0|+|\check{\pi}_1|+\dots+|\check{\pi}_N|}=q^{|\check{\pi}|},\\
\label{e:com1}&\prod^K_{i=1}\prod^N_{j=1}(1-x_iy_j)=\prod^K_{i=1}\prod^N_{j=1}(1-q^{K+N-i-j+2})=\prod^K_{i=1}\prod^N_{j=1}(1-q^{i+j}),
\end{align}
where $\pi_K=\emptyset,~\check{\pi}_N=\emptyset$. Then by \eqref{e:ortho1}
\begin{align}
\notag\langle 0|\phi^-(x_1)\dots \phi^-(x_K)\phi^+(y_1)\dots \phi^+(y_N)|0\rangle=&\sum_{\lambda\subseteq [K\times\infty]}\sum_{\{\lambda\rightarrow \pi\}_K}q^{|\pi|}\sum_{\mu\subseteq [N\times\infty]}\sum_{\{\mu\rightarrow \check{\pi}\}_N}q^{|\check{\pi}|}\langle\lambda|\mu\rangle\\
=&\sum_{\lambda\subseteq [K\times N]}(-1)^{|\lambda|}\sum_{\{\lambda\rightarrow \pi\}_K}\sum_{\{\lambda^{\prime}\rightarrow \check{\pi}\}_N}q^{|\pi|+|\check{\pi}|},
\end{align}
where $\lambda^{\prime}$ is the conjugate of $\lambda$. Taking the limit $K,N\rightarrow \infty$, we have the following result.
\begin{proposition} One has the identity:
\begin{align}
\sum_{\lambda\in\mathcal P}(-1)^{|\lambda|}\sum_{\{\lambda\rightarrow \pi\}}\sum_{\{\lambda^{\prime}\rightarrow \check{\pi}\}}q^{|\pi|+|\check{\pi}|}=\prod^\infty_{i=1}(1-q^i)^{i-1},
\end{align}
where $\{\lambda\rightarrow \pi\}$ (resp. $\{\lambda^{\prime}\rightarrow \check{\pi}\}$)
runs through all half plane partitions $\pi$ (resp. $\check{\pi}$) starting at the partition $\lambda$ (resp. $\lambda^{\prime}$).
\end{proposition}

\bigskip
\centerline{\bf Acknowledgments}
We would like to thank the anonymous referees for helpful comments which have made the paper more readable.
The work is partially supported by
Simons Foundation grant no. 523868 and NSFC grant no. 12171303.

\end{document}